\documentclass[10pt; a4paper]{amsart} %\documentclass[选项表]{类}[版本号yyyy/mm/dd]
\usepackage{amscd, amsfonts, amsmath, amssymb, amsthm, arydshln, fixmath, graphicx, mathrsfs, overpic}
\usepackage{hyperref}
\usepackage[all]{xy} %\usepackage[选项表]{宏包}[版本号yyyy/mm/dd] 数学字体宏包:mathrsfs, amsfonts, amssymb
\usepackage{tikz}
\usepackage{pgfplots}
\usepackage{pgflibraryarrows}
\usepackage{pgflibrarysnakes}
\makeindex

\theoremstyle{definition} %标题与编号为黑体, 正文为正常字体
\newtheorem{Unity}{Unity}[section] %\newtheorem{定理环境名}{标题}[主计数器名]
\newtheorem*{Definition*}{Definition} %\newtheorem*{定理环境名}[已定义定理环境名]{标题} 手动编号, 不自动编号
\newtheorem{Definition}[Unity]{Definition} %\newtheorem{定理环境名}[已定义定理环境名]{标题} 与当前环境共用同一个序号计数器

\theoremstyle{plain} %标题与编号为黑体, 正文为斜体
\newtheorem*{Theorem*}{Theorem}
\newtheorem{Theorem}[Unity]{Theorem}
\newtheorem{Proposition}[Unity]{Proposition}
\newtheorem{Corollary}[Unity]{Corollary}
\newtheorem{Lemma}[Unity]{Lemma}

\theoremstyle{remark} %标题与编号为斜体, 正文为正常字体
\newtheorem*{Remark*}{Remark}

\numberwithin{Unity}{section}%\numberwithin{计数器}{主计数器}

\newcommand{\Q}{\mathbb{Q}}
\newcommand{\PP}{\mathbb{P}}
\newcommand{\E}{\mathscr{E}}
\newcommand{\Ec}{\mathcal{E}}
\newcommand{\Lc}{\mathcal{L}}

\newcommand{\F}{\mathscr{F}}
\newcommand{\G}{\mathscr{G}}
\newcommand{\I}{\mathscr{I}}
\newcommand{\Ls}{\mathscr{L}}
\newcommand{\M}{\mathfrak{M}}

\newcommand{\Ox}{\mathscr{O}}
\newcommand{\Ps}{\mathscr{P}}
\newcommand{\Id}{\mathrm{Id}}
\newcommand{\rk}{\mathrm{rk}}
\newcommand{\Deg}{\mathrm{deg}}
\newcommand{\HNP}{\mathrm{HNP}}
\newcommand{\Jac}{\mathrm{Jac}}
\newcommand{\Pic}{\mathrm{Pic}}
\newcommand{\Frob}{\mathrm{Frob}}
\newcommand{\Quot}{\mathrm{Quot}}
\newcommand{\Spec}{\mathrm{Spec}}
\newcommand{\Supp}{\mathrm{Supp}}
\newcommand{\Omg}{\mathrm{\Omega}}

\newcommand{\Vect}{\mathfrak{Vect}}
\newcommand{\ConPgn}{\mathfrak{ConPgn}}

\begin{document}

\title{Frobenius Stratification of Moduli Spaces of Rank $3$ Vector Bundles in Characteristic $3$, I}
\author{Lingguang Li}
\address{School of Mathematical Sciences, Tongji University, Shanghai, P. R. China}
\email{LiLg@tongji.edu.cn}
\thanks{Partially supported by National Natural Science Foundation of China (Grant No. 11501418) and Shanghai Sailing Program.}

\begin{abstract}
Let $X$ be a smooth projective curve of genus $g\geq 2$ over an algebraically closed field $k$ of
characteristic $p>0$, $F_X:X\rightarrow X$ the absolute Frobenius morphism. Let $\M^s_X(r,d)$ be the moduli space of stable vector bundles of rank $r$ and degree $d$ on $X$. We study the Frobenius stratification of $\M^s_X(3,0)$ in terms of Harder-Narasimhan polygons of Frobenius pull backs of stable vector bundles and obtain the irreducibility and dimension of each non-empty Frobenius stratum in the case $(p,g)=(3,2)$.
\end{abstract}
\maketitle
$$Dedicated~To~The~Memory~of~Professor~Michel~Raynaud.$$
%\tableofcontents

\section{Introduction}

Let $k$ be an algebraically closed field of characteristic $p>0$, $X$ a smooth projective curve of genus $g$ over $k$. The absolute Frobenius morphism $F_X:X\rightarrow X$ is induced by $\Ox_X\rightarrow \Ox_X$, $f\mapsto f^p$. Let $\M^s_X(r,d)$ (resp. $\M^{ss}_X(r,d)$) be the moduli space of stable (resp. semistable) vector bundles of rank $r$ and degree $d$ on $X$.

Let $\E$ be a vector bundle on $X$, and
$$\mathrm{HN}_\bullet(\E):0=\E_0\subset\E_1\subset\cdots\subset\E_{m-1}\subset\E_m=\E$$
the Harder-Narasimhan filtration of $\E$. Consider the points
$$(\rk(\E_i),\Deg(\E_i))(0\leq i\leq m)$$ in the coordinate plane of $\mathrm{rank}$-$\mathrm{degree}$, we connect the point $(\rk(\E_i),\Deg(\E_i))$ to the point $(\rk(\E_{i+1}),\Deg(\E_{i+1}))$ successively by line segment for $0\leq i\leq m-1$. Then we get a convex polygon in the plane which we call the \emph{Harder-Narasimhan Polygon} of $\E$, denoted by $\HNP(\E)$.

Let $(r,d)$ be a point in the coordinate plane of $\mathrm{rank}$-$\mathrm{degree}$. If $r\leq\rk(\E)$, and $d\geq(\leq)d^\prime$ for some point $(r,d^\prime)\in\HNP(\E)$, then we say $(r,d)$ \emph{lies above $($below$)$} the $\HNP(\E)$. Given two convex polygons $\Ps_1$ and $\Ps_2$, if for any vertex on $\Ps_1$ lies above $\Ps_2$, then we say $\Ps_1$ lies above $\Ps_2$, denoted by $\Ps_1\succcurlyeq\Ps_2$. There is a natural partial order structure, in the sense of $\succcurlyeq$, on the set $\{\HNP(\E)~|~\E\in\Vect_X(r,d)~\}$, where $\Vect_X(r,d)$ is the category of vector bundles of rank $r$ and degree $d$ on $X$.

In general, the semistability of vector bundles is possibly destabilized under Frobenius pull back $F^*_X$ (cf. \cite{Gieseker73}, \cite{Raynaud82}). Thus, there is a natural set-theoretic map
\begin{eqnarray*}
S^s_{\Frob}:\M^s_X(r,d)(k)&\rightarrow&\ConPgn(r,pd)\\
{[\E]}&\mapsto&\HNP(F^*_X(\E))
\end{eqnarray*}
where $\ConPgn(r,pd)$ is the partially ordered set of all convex polygons in the coordinate plane such that their vertexes have integral coordinates, start at the origin $(0,0)$ and end at the point $(r,pd)$. For any $\Ps\in\ConPgn(r,pd)$, we denote
\begin{eqnarray*}
S_X(r,d,\Ps)&:=&\{[\E]\in\M^s_X(r,d)(k)~|~\HNP(F^*_X(\E))=\Ps\}.\\
S_X(r,d,\Ps^+)&:=&\{[\E]\in\M^s_X(r,d)(k)~|~\HNP(F^*_X(\E))\succcurlyeq\Ps\}.
\end{eqnarray*}
Then we have a canonical stratification of $\M^s_X(r,d)$ by Harder-Narasimhan polygons of Frobenius pull backs of stable vector bundles of rank $r$ and degree $d$. We call this the \emph{Frobenius stratification}.
By a theorem of S. S. Shatz \cite[Theorem 3]{Shatz77} and the geometric invariant theory construction of $\M^s_X(r,d)$, we know that the Frobenius stratum $S_X(r,d,\Ps^+)$ is a closed subvariety of $\M^s_X(r,d)$ for any $\Ps\in\ConPgn(r,pd)$.

The main goal of the paper is to study the geometric properties of Frobenius strata, such as non-emptiness, irreducibility, connectedness, smoothness, dimension and so on. Some results are known in special cases for small values of $p$, $g$, $r$ and $d$. In particular, Joshi-Ramanan-Xia-Yu \cite{JRXY06} give a complete description of the Frobenius stratification of $\M^s_{X}(2,d)$ for any integer $d$ when $p=2$ and $g\geq 2$. They obtain the irreducibility and dimension of each non-empty Frobenius stratum. Fix integers $r$ and $d$ with $r>0$, we denote
$$W_X(r,d):=\{[\E]\in\M^s_{X}(r,d)(k)|\HNP(F^*_X(\E))\succcurlyeq\HNP(F^*_X(\F))~\text{for any}~[\F]\in\M^s_{X}(r,d)(k)\}.$$
Then $W_X(r,d)$ is a closed subvariety of $\M^s_{X}(r,d)$. Joshi-Ramanan-Xia-Yu \cite[Theorem 4.6.4]{JRXY06} show that $W_X(2,d)$ is irreducible and of dimension $g$ when $p=2$ and $g\geq 2$. Moreover, under the assumption $p>r(r-1)(g-1)$, Joshi and Pauly \cite[Theorem 6.2.1]{JoshiPauly15} show that the dimension of $W_X(r,0)$ is $g$. Other results about Frobenius stratification can be found in \cite{JX00}\cite{LanP08}\cite{LasP02}\cite{LasP04}\cite{Oss06}\cite{Zh17}.

In \cite{Li14} the author shows that $W_X(p,d)$ is a closed subvariety of $\M^s_{X}(p,d)$ which is isomorphic to the Jacobian variety $\Jac_X$ of $X$ for any integer $d$ (cf. \cite[Theorem 3.2]{Li14}). In particular, $W_X(p,d)$ is an smooth irreducible projective variety of dimension $g$. Combining \cite[Theorem 1.1]{LiuZhou13} with \cite[Theorem 2.5]{Li14}, we can obtain the geometric properties of a specific Frobenius stratum.

\begin{Theorem}$($Theorem \ref{GeneralStratum}$)$
Let $k$ be an algebraically closed field of characteristic $p>0$, $X$ a smooth projective curve of genus $g\geq 2$ over $k$. Then the subset $$V_{rp,d}=\{[\E]\in\M^s_X(rp,d)(k)~|~\HNP(F^*_X(\E))=\Ps^{\mathrm{can}}_{rp,d}\}$$
is a smooth irreducible closed subvariety of dimension $r^2(g-1)+1$ in $\M^s_X(rp,d)$.
\end{Theorem}

In this paper, we mainly study the Frobenius stratification of $\M^s_X(3,0)$, where $X$ is a smooth projective curve of genus $2$ over an algebraically closed field $k$ of characteristic $3$. In this case, there are $4$ possible Harder-Narasimhan polygons $\{\Ps_i\}_{1\leq i\leq 4}$ for Frobenius pull backs of Frobenius destabilized semistable vector bundles of rank $3$ and degree $0$ (See Section $2$). We obtain the irreducibility and dimension of each non-empty Frobenius stratum in the moduli space $\M^s_X(3,0)$ when $(p,g)=(3,2)$. The structure of Frobenius stratification of $\M^s_X(3,d)$ is easily deduced from $\M^s_X(3,0)$, when $3|d$.

In general, it is difficult to determine the Harder-Narasimhan polygon of $F^*_X(\E)$ for a stable bundle $\E$ on $X$. We first show that any rank $3$ and degree $0$ Frobenius destabilized stable vector bundle $\E$ with $\HNP(F^*_X(\E))\in\{\Ps_2,\Ps_3,\Ps_4\}$ can be embedded into ${F_X}_*(\Ls)$ for some line bundle $\Ls$ of degree $-1$, when $(p,g)=(3,2)$ (Proposition \ref{Prop:Injection}). Then we can determine $\HNP(F^*_X(\E))$ by analysing the intersection of $F^*_X(\E)$ with the canonical filtration of $F^*_X({F_X}_*(\Ls))$. Moreover, we show that any rank $3$ and degree $0$ subsheaf $\E$ of ${F_X}_*(\Ls)$ is semistable for any line bundle $\Ls$ of degree $-1$ (Proposition \ref{Prop:Subsheaf}). Then we have a morphism
\begin{eqnarray*}
\theta:\Quot_X(3,0,\Pic^{(-1)}(X))&\rightarrow&\M^{ss}_X(3,0)\\
{[\E\hookrightarrow{F_X}_*(\Ls)]}&\mapsto&[\E],
\end{eqnarray*}
where $\Quot_X(3,0,\Pic^{(-1)}(X))$ is the Quot scheme parameterizing all rank $3$ and degree $0$ subsheaves of ${F_X}_*(\Ls)$ for any line bunde $\Ls\in\Pic^{(-1)}(X)$. Restricting the morphism $\theta$ to the stable locus $\Quot^s_X(3,0,\Pic^{(-1)}(X))$ of $\Quot_X(3,0,\Pic^{(-1)}(X))$, we can obtain the geometric properties of Frobenius strata of $\M^s_X(3,0)$ from the geometric properties of Frobenius strata of $\Quot^s_X(3,0,\Pic^{(-1)}(X))$. These techniques are generalizations of the methods which are first introduced by Joshi-Ramanan-Xia-Yu in \cite{JRXY06}.

The main result of this paper is the following Theorem.

\begin{Theorem}$($Theorem \ref{Thm:FrobStra}$)$
Let $k$ be an algebraically closed field of characteristic $3$, $X$ a smooth projective curve of genus $2$ over $k$. Then
\begin{itemize}
    \item[$(1)$] $S_X(3,0,\Ps^+_1)\cong S_X(3,0,\Ps^+_2)$, $S_X(3,0,\Ps_1)\cong S_X(3,0,\Ps_2)$, and $$S_X(3,0,\Ps^+_1)\cap S_X(3,0,\Ps^+_2)=S_X(3,0,\Ps^+_3).$$
    \item[$(2)$] $S_X(3,0,\Ps^+_i)=\overline{S_X(3,0,\Ps_i)}$, $S_X(3,0,\Ps_i)$ and $S_X(3,0,\Ps^+_i)$ are irreducible quasi-projective varieties for $1\leq i\leq 4$, and $$\dim S_X(3,0,\Ps^+_i)=\dim S_X(3,0,\Ps_i)=
\begin{cases}
5, \mathrm{when}~i=1\\
5, \mathrm{when}~i=2\\
4, \mathrm{when}~i=3\\
2, \mathrm{when}~i=4\\
\end{cases}$$
\end{itemize}
\end{Theorem}

In section 2, we show that there are $4$ possible Harder-Narasimhan polygons $\{\Ps_1,\Ps_2,\Ps_3,\Ps_4\}$ for Frobenius destabilized semistable vector bundles of rank $3$ and degree $0$ in the case $(p,g)=(3,2)$.

In section 3, we show that any Frobenius destabilized stable bundle $\E$ of rank $3$ and degree $0$ with $\HNP(F^*_X(\E))\in\{\Ps_2,\Ps_3,\Ps_4\}$ can be embedded into ${F_X}_*(\Ls)$ for some line bundle $\Ls$ of degree $-1$. Moreover, we show that for any line bundle $\Ls$ of degree $-1$ on $X$, each rank $3$ and degree $0$ subsheaf $\E\subset {F_X}_*(\Ls)$ is semistable.

In section 4, we will study the Frobenius stratification of the Quot scheme $\Quot_X(3,d,\Pic^{(d-1)}(X))$ and obtain the smoothness, irreducibility and dimension of each non-empty stratum.

In section 5, we study the Frobenius stratification of moduli space $\M^s_X(3,0)$ when $(p,g)=(3,2)$. We obtain the geometric properties of Frobenius strata in $\M^s_X(3,0)$ from the geometric properties of Frobenius strata in $\Quot^s_X(3,0,\Pic^{(-1)}(X))$ by the morphism $\theta^s:\Quot^s_X(3,0,\Pic^{(-1)}(X))\rightarrow\M^s_X(3,0):{[\E\hookrightarrow{F_X}_*(\Ls)]}\mapsto[\E]$. Moreover, we obtain the geometric properties of a special Frobenius stratum in $\M^s_{X}(rp,d)$ for any integers $p$, $g$, $r$, $d$ with $r>0$, $p>0$ and $g\geq2$.

\section{Classification of Frobenius Harder-Narasimhan Polygons}

In this section, we will determine all of the possible Harder-Narasimhan polygons of $F^*_X(\E)$ for any Frobenius destabilized semistable vector bundles $\E\in\M^s_X(3,0)$, where $X$ is a smooth projective curve of genus $2$ over an algebraically closed field $k$ of characteristic $3$.
\begin{Lemma}[N. I. Shepherd-Barron \cite{Shepherd-Barron98} and V. Mehta, C. Pauly \cite{MehtaPauly07}]\label{Lemma:InsHN}
Let $k$ be an algebraically closed field of characteristic $p>0$, $X$ a smooth projective curve of genus $g\geq 2$ over $k$, $\E$ a semistable vector bundle on $X$. Let $0=\E_0\subset\E_1\subset\cdots\subset\E_{m-1}\subset\E_m=F^*_X(\E)$ be the Harder-Narasimhan filtration of $F^*_X(\E)$. Then
\begin{itemize}
\item[$(1)$] For any $1\leq i\leq m-1$, $\mu(\E_i/\E_{i-1})-\mu(\E_{i+1}/\E_i)\leq 2g-2$.
\item[$(2)$] $\mu_{\max}(F^*_X(\E))-\mu_{\min}(F^*_X(\E))\leq\min\{r-1,p-1\}\cdot(2g-2)$.
\end{itemize}
\end{Lemma}

According to the Lemma \ref{Lemma:InsHN}, in the case $(p,g,r,d)=(3,2,3,0)$, there are $4$ possible Harder-Narasimhan polygons for all Frobenius destabilized semistable vector bundles of rank $3$ and degree $0$ when $(p,g)=(3,2)$ as the following:

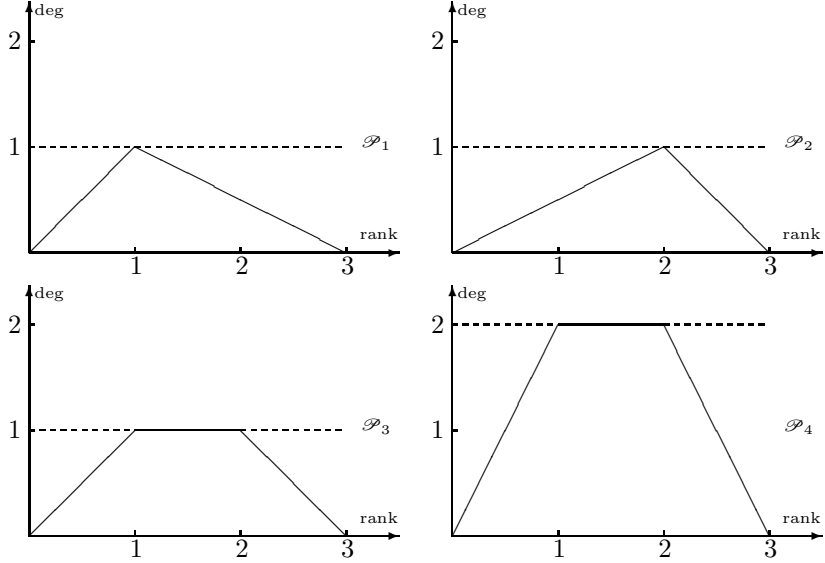
\begin{center}
\begin{tabular}{cc}
%P_1
\begin{picture}(150,100)
\linethickness{0.5pt}
% 坐标轴
\put(0,0){\vector(1,0){140}} %坐标横轴
\put(0,0){\vector(0,1){95}} %坐标纵轴
\multiput(0,0)(40,0){4}{\line(0,1){2}}%坐标横轴单位坐标
\put(38,-8){1}
\put(78,-8){2}
\put(118,-8){3}
\multiput(0,0)(0,40){3}{\line(1,0){2}}%坐标纵轴单位坐标
\put(-8,37){1}
\put(-8,77){2}
\put(125,5){\tiny{rank}}
\put(2,90){\tiny{deg}}
\thinlines
\multiput(0,40)(4,0){30}{\line(1,0){2}} %顶部虚线
\put(0,0){\line(1,1){40}} %最左实线
\put(40,40){\line(2,-1){80}}%最右实线
%\put(100,60){\shortstack{\tiny{Harder-Narasimhan}\\\tiny{filtration of $F^*_X(\E)$}}}
\put(125,40){\tiny{$\Ps_1$}}
\end{picture}
&
%P_2
\begin{picture}(150,100)
\linethickness{0.5pt}
% 坐标轴
\put(0,0){\vector(1,0){140}} %坐标横轴
\put(0,0){\vector(0,1){95}} %坐标纵轴
\multiput(0,0)(40,0){4}{\line(0,1){2}}%坐标横轴单位坐标
\put(38,-8){1}
\put(78,-8){2}
\put(118,-8){3}
\multiput(0,0)(0,40){3}{\line(1,0){2}}%坐标纵轴单位坐标
\put(-8,37){1}
\put(-8,77){2}
\put(125,5){\tiny{rank}}
\put(2,90){\tiny{deg}}
\thinlines
\multiput(0,40)(4,0){30}{\line(1,0){2}} %顶部虚线
\put(0,0){\line(2,1){80}} %最左实线
\put(80,40){\line(1,-1){40}}%最右实线
%\put(100,60){\shortstack{\tiny{Harder-Narasimhan}\\\tiny{filtration of $F^*_X(\E)$}}}
\put(125,40){\tiny{$\Ps_2$}}
\end{picture}
\end{tabular}
\end{center}

\begin{center}
\begin{tabular}{cc}
%P_3
\begin{picture}(150,100)
\linethickness{0.5pt}
% 坐标轴
\put(0,0){\vector(1,0){140}} %坐标横轴
\put(0,0){\vector(0,1){95}} %坐标纵轴
\multiput(0,0)(40,0){4}{\line(0,1){2}}%坐标横轴单位坐标
\put(38,-8){1}
\put(78,-8){2}
\put(118,-8){3}
\multiput(0,0)(0,40){3}{\line(1,0){2}}%坐标纵轴单位坐标
\put(-8,37){1}
\put(-8,77){2}
\put(125,5){\tiny{rank}}
\put(2,90){\tiny{deg}}
\thinlines
\multiput(0,40)(4,0){30}{\line(1,0){2}} %顶部虚线
\put(0,0){\line(1,1){40}} %最左实线
\put(40,40){\line(1,0){40}} %上层实线
\put(80,40){\line(1,-1){40}}%最右实线
%\put(100,60){\shortstack{\tiny{Harder-Narasimhan}\\\tiny{filtration of $F^*_X(\E)$}}}
\put(125,40){\tiny{$\Ps_3$}}
\end{picture}
&
%P_4
\begin{picture}(150,100)
\linethickness{0.5pt}
% 坐标轴
\put(0,0){\vector(1,0){140}} %坐标横轴
\put(0,0){\vector(0,1){95}} %坐标纵轴
\multiput(0,0)(40,0){4}{\line(0,1){2}}%坐标横轴单位坐标
\put(38,-8){1}
\put(78,-8){2}
\put(118,-8){3}
\multiput(0,0)(0,40){3}{\line(1,0){2}}%坐标纵轴单位坐标
\put(-8,37){1}
\put(-8,77){2}
\put(125,5){\tiny{rank}}
\put(2,90){\tiny{deg}}
\thinlines
\multiput(0,80)(4,0){30}{\line(1,0){2}} %顶部虚线
\put(0,0){\line(1,2){40}} %最左实线
\put(40,80){\line(1,0){40}} %上层实线
\put(80,80){\line(1,-2){40}}%最右实线
%\put(100,60){\shortstack{\tiny{Harder-Narasimhan}\\\tiny{filtration of $F^*_X(\E)$}}}
\put(125,40){\tiny{$\Ps_4$}}
\end{picture}
\end{tabular}
\end{center}

\section{Construction of Stable Vector Bundles}

\begin{Definition}[Joshi-Ramanan-Xia-Yu \cite{JRXY06}]\label{CanFil}
Let $k$ be an algebraically closed field of characteristic $p>0$, $X$ a smooth projective curve over $k$. For
any coherent sheaf $\F$ on $X$, let
$$\nabla_{\mathrm{can}}:F^*_X{F_X}_*(\F)\rightarrow
F^*_X{F_X}_*(\F)\otimes_{\Ox_X}\Omg^1_X$$
be the canonical connection on $F^*_X{F_X}_*(\F)$.
Set
\begin{eqnarray*}
V_1&:=&\ker(F^*_X{F_X}_*(\F)\twoheadrightarrow\F),\\
V_{l+1}&:=&\ker\{V_l\stackrel{\nabla_{\mathrm{can}}}{\longrightarrow}F^*_X{F_X}_*(\F)
\otimes_{\Ox_X}\Omg^1_X\rightarrow (F^*_X{F_X}_*(\F)/V_l)
\otimes_{\Ox_X}\Omg^1_X\}
\end{eqnarray*}
The filtration
$${\mathbb{F}^{\mathrm{can}}_\F}_\bullet:F^*_X{F_X}_*(\F)=V_0\supset V_1\supset V_2\supset\cdots\supset V_{p-1}\supset V_p=0$$
is called the \emph{canonical filtration} of
$F^*_X{F_X}_*(\F)$.
\end{Definition}

\begin{Lemma}[Joshi-Ramanan-Xia-Yu \cite{JRXY06} and X. Sun \cite{Sun08i}]\label{Lem:Sun}
Let $k$ be an algebraically closed field of characteristic $p>0$, $X$ a smooth projective curve of genus $g$ over $k$, $\E$ a vector bundle on $X$. Then the canonical filtration of $F^*_X{F_X}_*(\E)$ is $$0=V_p\subset V_{p-1}\subset\cdots\subset V_{l+1}\subset V_l\subset\cdots\subset V_1\subset V_0=F^*_X{F_X}_*(\E)$$
such that
\begin{itemize}
\item[$(1)$] $\nabla_{\mathrm{can}}(V_{i+1})\subset V_i\otimes_{\Ox_X}\Omg^1_X$ for $0\leq i\leq p-1$.
\item[$(2)$] $V_l/V_{l+1}\stackrel{\nabla_{\mathrm{can}}}{\longrightarrow}\E\otimes_{\Ox_X}\Omg^{\otimes l}_X$ are isomorphic for $0\leq l\leq p-1$.
\item[$(3)$] If $g\geq 1$, then ${F_X}_*(\E)$ is semistable whenever $\E$ is semistable. If $g\geq 2$, then ${F_X}_*(\E)$ is stable whenever $\E$ is stable.
\item[$(4)$] If $g\geq 2$ and $\E$ is semistable, then the canonical filtration of $F^*_X{F_X}_*(\E)$ is nothing but the Harder-Narasimhan filtration of $F^*_X{F_X}_*(\E)$.
\end{itemize}
\end{Lemma}

\begin{Proposition}\label{Prop:Injection}
Let $k$ be an algebraically closed field of characteristic $3$, $X$ a smooth projective curve of genus $2$ over $k$. Let $\E$ be a rank $3$ and degree $0$ stable vector bundle on $X$ and one has non-trivial homomorphism $F^*_X(\E)\rightarrow\Ls$, where $\Ls$ is a line bundle on $X$ with $\deg(\Ls)=-1$. Then the adjoint homomorphism $\E\hookrightarrow{F_X}_*(\Ls)$ is an injection.
\end{Proposition}

\begin{proof}
By adjunction, there is a non-trivial homomorphism $\E\rightarrow{F_X}_*(\Ls)$. Denote the image by $\G$. By stability of $\E$, we have $\deg(\G)\geq0$ and $\deg(\G)=0$ if and only if $\E\cong\G$. On the other hand, we have $\deg({F_X}_*(\Ls))=1$ (cf. \cite[Sect. 2.9]{JRXY06}). Moreover, by Lemma \ref{Lem:Sun}(3), the stability of ${F_X}_*(\Ls)$ implies that $\deg(\G)\leq0$. Hence $\deg(\G)=0$. Thus, $\E\cong\G$, i.e. the adjoint homomorphism $\E\hookrightarrow{F_X}_*(\Ls)$ is an injection.
\end{proof}

\begin{Proposition}\label{Prop:Subsheaf}
Let $k$ be an algebraically closed field of characteristic $3$, $X$ a smooth projective curve of genus $2$ over $k$. Let $\Ls$ be a line bundle on $X$ with $\deg(\Ls)=-1$, $\E$ a subsheaf of $F^*_X(\Ls)$ with $\rk(\E)=3$ and $\deg(\E)=0$. Then $\E$ is a semistable vector bundle. Moreover, there exists some rank $3$ and degree $0$ stable subsheaf of ${F_X}_*(\Ls)$.
\end{Proposition}

\begin{proof}
Let $\G\subset\E$ be a subsheaf of $\E$ with $\rk(\G)<\rk(\E)=3$. By \cite[Corollary 2.4]{Sun10ii} and the stability of ${F_X}_*(\Ls)$, we have
$$\mu(\G)-\mu({F_X}_*(\Ls))\leq-\frac{p-\rk(\G)}{p}(g-1)=-\frac{3-\rk(\G)}{3}.$$
It follows that
$$\mu(\G)\leq-\frac{2-\rk(\G)}{3}\leq 0.$$
Thus $\E$ is a semistable vector bundle.

Let $x$ be a closed point of $X$. Then ${F_X}_*(\Ls(-x))$ is a rank $3$ and degree $0$ stable subsheaf of ${F_X}_*(\Ls)$ by \cite[Sect. 2.9]{JRXY06} and Lemma \ref{Lem:Sun}(3).
\end{proof}

By Proposition \ref{Prop:Injection}, Proposition \ref{Prop:Subsheaf} and the classification of Harder-Narasimhan polygons of Frobenius pull backs of Frobenius destabilized semistable vector bundles in the case $(p,g,r,d)=(3,2,3,0)$, we have any Frobenius destabilized stable bundle $\E\in\M^s_{X}(3,0)$ with $\HNP(F^*_X(\E))\neq\Ps_1$ can be embedded into $F^*_X(\Ls)$ for some line bundle $\Ls$ on $X$ of degree $-1$.

\section{Frobenius Stratification of Quot Schemes}

Let $k$ be an algebraically closed field of characteristic $p>0$, $X$ a smooth projective curve of genus $g$ over $k$, $F_X:X\rightarrow X$ the absolute Frobenius morphism. Let $P^t(T)\in\Q[T]$ be the Hilbert polynomial of ${F_X}_*(\Ls)$ for any line bundle $\Ls$ of degree $t$ on $X$, and $P_{r,d}(T)\in\Q[T]$ the Hilbert polynomial of any vector bundle $\F$ of rank $r$ and degree $d$ on $X$. Denote $$\Phi^t_{r,d}(T):=P^t(T)-P_{r,d}(T)\in\Q[T].$$

Let $\Pic^{(t)}(X)$ be the Picard scheme parameterizing all line bundles of degree $t$ on $X$, $\Lc$ the universal line bundle on $\Pic^{(t)}(X)\times_kX$, and
$$\pi:\Quot_X(r,d,\Pic^{(t)}(X)):=\Quot^{\Phi^t_{r,d}}_{(\Id_{\Pic^{(t)}(X)}\times F_X)_*(\Lc)/\Pic^{(t)}(X)\times_kX/\Pic^{(t)}(X)}\rightarrow\Pic^{(t)}(X)$$ the Quot scheme classifying the subsheaves of $(\Id_{\Pic^{(t)}(X)}\times F_X)_*(\Lc)$ on $\Pic^{(t)}(X)\times_kX$, which are flat families of vector bundles of rank $r$ and degree $d$ on $X$ parameterized by $\Pic^{(t)}(X)$. Then for any $\Ls\in\Pic^{(t)}(X)$, the fiber of $\pi$ over $[\Ls]$ is the Quot scheme $\Quot_X(r,d,\Ls):=\Quot^{\Phi^t_{r,d}}_{{F_X}_*(\Ls)/X/k}$ which parameterizing all rank $r$ and degree $d$ subsheaves of ${F_X}_*(\Ls)$ on $X$.

Consider the commutative diagram of morphisms
$$\xymatrix{
\Quot_X(r,d,\Pic^{(t)}(X))\times_kX \ar[r]^-{\Id_{\Quot}\times F_X}\ar[d]^{\pi\times\Id_X} & \Quot_X(r,d,\Pic^{(t)}(X))\times_kX \ar[d]^{\pi\times\Id_X} \\
\Pic^{(t)}(X)\times_kX \ar[r]^-{\Id_{\Pic^{(t)}(X)}\times F_X} & \Pic^{(t)}(X)\times_kX.}$$
and the universal subsheaf $$\Ec\hookrightarrow(\pi\times\Id_X)^*(\Id_{\Quot}\times F_X)_*(\Lc)=(\Id_{\Quot}\times F_X)_*(\pi\times\Id_X)^*(\Lc)$$
on $\Quot_X(r,d,\Pic^{(t)}(X))\times_kX$. By adjunction, we have homomorphism
$${(\Id_{\Quot}\times F_X)}^*(\Ec)\rightarrow(\pi\times\Id_X)^*(\Lc)$$
on $\Quot_X(r,d,\Pic^{(t)}(X))\times_kX$.
Let $\G$ be the co-kernel of above homomorphism. Then $\mathrm{pr}_*(\G)$ is a coherent sheaf on $\Quot_X(r,d,\Pic^{(t)}(X))$, where $$\mathrm{pr}:\Quot_X(r,d,\Pic^{(t)}(X))\times_kX\rightarrow\Quot_X(r,d,\Pic^{(t)}(X))$$ is the natural projection. Then the subset
$$\Quot^\sharp_X(r,d,\Pic^{(t)}(X)):=\{q\in\Quot_X(r,d,\Pic^{(t)}(X))|\dim_{\kappa(q)}(\mathrm{pr}_*(\G)_q\otimes_{\Ox_q}\kappa(q))=0\}$$
is an open subscheme of $\Quot_X(r,d,\Pic^{(t)}(X))$. Then we have
\begin{eqnarray*}
\Quot_X(r,d,\Pic^{(t)}(X))(k)&=&\{~[\E\hookrightarrow{F_X}_*(\Ls)]~|~\rk(\E)=r,\deg(\E)=d,\Ls\in\Pic^{(t)}(X)~\}\\
\Quot^\sharp_X(r,d,\Pic^{(t)}(X))(k)&=&\left\{~[\E\hookrightarrow{F_X}_*(\Ls)]\in\Quot_X(r,d,\Pic^{(t)}(X))(k)~\Bigg|
\begin{array}{ll}
\text{The adjoint homomorphism}\\
F^*_X(\E)\rightarrow\Ls~\text{is surjective}.
\end{array}
\right\}
\end{eqnarray*}

Let $\Ps\in\ConPgn(r,pd)$, we denote the subschemes of $\Quot_X(r,d,\Pic^{(t)}(X))$ and $\Quot^\sharp_X(r,d,\Pic^{(t)}(X))$ as the following (For simplicity, we describe these Quot schemes in the sense of closed points):
\begin{eqnarray*}
\Quot_X(r,d,\Pic^{(t)}(X),\Ps)(k)&:=&\{~[\E\hookrightarrow{F_X}_*(\Ls)]\in\Quot_X(r,d,\Pic^{(t)}(X))(k)~|~\HNP(F^*_X(\E))=\Ps~\}\\
\Quot_X(r,d,\Pic^{(t)}(X),\Ps^+)(k)&:=&\{~[\E\hookrightarrow{F_X}_*(\Ls)]\in\Quot_X(r,d,\Pic^{(t)}(X))(k)~|~\HNP(F^*_X(\E))\succcurlyeq\Ps~\}\\
\Quot^\sharp_X(r,d,\Pic^{(t)}(X),\Ps)(k)&:=&\{~[\E\hookrightarrow{F_X}_*(\Ls)]\in\Quot^\sharp_X(r,d,\Pic^{(t)}(X))(k)~|~\HNP(F^*_X(\E))=\Ps~\}\\
\Quot^\sharp_X(r,d,\Pic^{(t)}(X),\Ps^+)(k)&:=&\{~[\E\hookrightarrow{F_X}_*(\Ls)]\in\Quot^\sharp_X(r,d,\Pic^{(t)}(X))(k)~|~\HNP(F^*_X(\E))\succcurlyeq\Ps~\}.
\end{eqnarray*}
Let $\Ls\in\Pic^{(t)}(X)$. The fibers of the projections
\begin{eqnarray*}
\Quot_X(r,d,\Pic^{(t)}(X))&\stackrel{\pi}{\rightarrow}&\Pic^{(t)}(X)\\
\Quot_X(r,d,\Pic^{(t)}(X),\Ps)\hookrightarrow\Quot_X(r,d,\Pic^{(t)}(X))&\stackrel{\pi}{\rightarrow}&\Pic^{(t)}(X)\\
\Quot_X(r,d,\Pic^{(t)}(X),\Ps^+)\hookrightarrow\Quot_X(r,d,\Pic^{(t)}(X))&\stackrel{\pi}{\rightarrow}&\Pic^{(t)}(X).
\end{eqnarray*}
\begin{eqnarray*}
\Quot^\sharp_X(r,d,\Pic^{(t)}(X))\hookrightarrow\Quot_X(r,d,\Pic^{(t)}(X))&\stackrel{\pi}{\rightarrow}&\Pic^{(t)}(X)\\
\Quot^\sharp_X(r,d,\Pic^{(t)}(X),\Ps)\hookrightarrow\Quot^\sharp_X(r,d,\Pic^{(t)}(X))\hookrightarrow\Quot_X(r,d,\Pic^{(t)}(X))&\stackrel{\pi}{\rightarrow}&\Pic^{(t)}(X)\\
\Quot^\sharp_X(r,d,\Pic^{(t)}(X),\Ps^+)\hookrightarrow\Quot^\sharp_X(r,d,\Pic^{(t)}(X))\hookrightarrow\Quot_X(r,d,\Pic^{(t)}(X))&\stackrel{\pi}{\rightarrow}&\Pic^{(t)}(X).
\end{eqnarray*}
over the $[\Ls]$ are denoted by $\Quot_X(r,d,\Ls)$, $\Quot_X(r,d,\Ls,\Ps)$, $\Quot_X(r,d,\Ls,\Ps^+)$, $\Quot^\sharp_X(r,d,\Ls)$, $\Quot^\sharp_X(r,d,\Ls,\Ps)$ and $\Quot^\sharp_X(r,d,\Ls,\Ps^+)$ respectively.
For example, the scheme $\Quot^\sharp_X(r,d,\Ls,\Ps)$ parameterizing all rank $r$ and degree $d$ subsheaves $\E\subset{F_X}_*(\Ls)$ with surjective adjoint homomorphism $F^*_X(\E)\rightarrow\Ls$ such that $\HNP(F^*_X(\E))=\Ps$.

In this section, we are interested in the Frobenius stratification of moduli spaces of vector bundles in the case $$(p,g,r,d,t)=(3,2,3,0,-1).$$
In this case, the scheme $\Quot_X(3,0,\Pic^{(-1)}(X))$ parameterizing all the rank $3$ and degree $0$ subsheaves of ${F_X}_*(\Ls)$ for any line bundle $\Ls$ of degree $-1$ on $X$. By Proposition \ref{Prop:Injection} and Proposition \ref{Prop:Subsheaf}, we know that these vector bundles are semistable. This induces a natural morphism
\begin{eqnarray*}
\theta:\Quot_X(3,0,\Pic^{(-1)}(X))&\rightarrow&\M^{ss}_X(3,0)\\
{[\E\hookrightarrow{F_X}_*(\Ls)]}&\mapsto&[\E].
\end{eqnarray*}
Now, we first analysis the structure of the $\Quot_X(3,0,\Pic^{(-1)}(X))$. Let $$e:=[\E\hookrightarrow{F_X}_*(\Ls)]$$
be a closed point of $\Quot_X(3,0,\Pic^{(-1)}(X))$, where $\Ls\in\Pic^{(-1)}(X)$. The non-trivial adjoint homomorphism $F^*_X(\E)\rightarrow\Ls$ implies that $$\mu(F^*_X(\E))>\mu(\Ls)\geq\mu_{\mathrm{min}}(F^*_X(\E)),$$ so $\E$ is a Frobenius destabilized semistable vector bundle.

\begin{Proposition}\label{Prop:Intersection}
Let $k$ be an algebraically closed field of characteristic $3$, $X$ a smooth projective curve of genus $2$ over $k$. Let $\Ls$ be a line bundle of degree $-1$ on $X$, $0\subset E_2\subset E_1\subset F^*_X({F_X}_*(\Ls))$ the canonical filtration of $F^*_X({F_X}_*(\Ls))$. Let $[\E\hookrightarrow{F_X}_*(\Ls)]\in\Quot_X(3,0,\Ls)(k)$. Then $\HNP(F^*_X(\E))\in\{\Ps_2,\Ps_3,\Ps_4\}$, and
\begin{itemize}
\item[$(1)$] $\HNP(F^*_X(\E))=\Ps_4$ if and only if $\deg(F^*_X(\E)\cap E_2)=2$ if and only if the adjoint homomorphism $F^*_X(\E)\rightarrow\Ls$ is not surjective. In this case, the Harder-Narasimhan filtration of $F^*_X(\E)$ is $$0\subset F^*_X(\E)\cap E_2\subset F^*_X(\E)\cap E_1\subset F^*_X(\E).$$
\item[$(2)$] $\HNP(F^*_X(\E))=\Ps_3$ if and only if $\deg(F^*_X(\E)\cap E_2)=1$. In this case, $[\E\hookrightarrow{F_X}_*(\Ls)]\in\Quot^\sharp_X(3,0,\Ls)(k)$ and the Harder-Narasimhan filtration of $F^*_X(\E)$ is $$0\subset F^*_X(\E)\cap E_2\subset F^*_X(\E)\cap E_1\subset F^*_X(\E).$$
\item[$(3)$] $\HNP(F^*_X(\E))=\Ps_2$ if and only if $\deg(F^*_X(\E)\cap E_2)=0$. In this case, $[\E\hookrightarrow{F_X}_*(\Ls)]\in\Quot^\sharp_X(3,0,\Ls)(k)$ and the Harder-Narasimhan filtration of $F^*_X(\E)$ is $$0\subset F^*_X(\E)\cap E_1\subset F^*_X(\E).$$
\end{itemize}
\end{Proposition}

\begin{proof}
The canonical filtration $0\subset E_2\subset E_1\subset F^*_X({F_X}_*(\Ls))$ induces the filtration $$0\subset F^*_X(\E)\cap E_2\subset F^*_X(\E)\cap E_1\subset F^*_X(\E).$$
Then the injections $F^*_X(\E)\cap E_2\hookrightarrow E_2$ and $F^*_X(\E)/(F^*_X(\E)\cap E_1)\hookrightarrow F^*_X({F_X}_*(\Ls))/E_1$ imply that $\deg(F^*_X(\E)\cap E_2)\leq 3$ and $\deg(F^*_X(\E)/(F^*_X(\E)\cap E_1))\leq -1$. Suppose that $[\E\hookrightarrow{F_X}_*(\Ls)]\in\Quot_X(3,0,\Ls)(k)$ such that $F^*_X(\E)$ is semistable or $\HNP(F^*_X(\E))=\Ps_1$, then $\mu_{\mathrm{min}}(F^*_X(\E))\geq-\frac{1}{2}$. This contradicts the fact $\deg(F^*_X(\E)/(F^*_X(\E)\cap E_1))\leq -1$. Hence $\HNP(F^*_X(\E))\in\{\Ps_2,\Ps_3,\Ps_4\}$.

We first claim that $$0\leq\deg(F^*_X(\E)\cap E_2)\leq 2.$$ The fact $\deg(F^*_X(\E)\cap E_2)\leq 2$ is followed by \cite[Theorem 2]{Shatz77} and the classification of Harder-Narasimhan polygons of Frobenius pull backs of Frobenius destabilized semistable vector bundles in the case $(p,g,r,d)=(3,2,3,0)$. On the other hand, suppose that $\deg(F^*_X(\E)\cap E_2)<0$, then $$\deg((F^*_X(\E)\cap E_1)/(F^*_X(\E)\cap E_2))\geq 2.$$ Then the injection
$$(F^*_X(\E)\cap E_1)/(F^*_X(\E)\cap E_2)\hookrightarrow(F^*_X(\E)/(F^*_X(\E)\cap E_1))\otimes_{\Ox_X}\Omg^1_X$$ induces a contradiction, since $\deg((F^*_X(\E)/(F^*_X(\E)\cap E_1))\otimes_{\Ox_X}\Omg^1_X)\leq 1$. This completes the proof of the claim.

(1).
I. If $\HNP(F^*_X(\E))=\Ps_4$, there exists a unique maximal destabilizing sub-line bundle $E'\subset F^*_X(\E)$ with $\deg(E')=2$. Suppose that $E'\nsubseteq F^*_X(\E)\cap E_1$, then the composition $$E'\hookrightarrow F^*_X(\E)\hookrightarrow F^*_X({F_X}_*(\Ls))\twoheadrightarrow F^*_X({F_X}_*(\Ls))/E_1\cong\Ls$$ is non-trivial. This induces a contradiction since $\deg(E')>\deg(\Ls)$. Suppose that $E'\subset F^*_X(\E)\cap E_1$ and $E'\nsubseteq F^*_X(\E)\cap E_2$, then the composition $$E'\hookrightarrow F^*_X(\E)\cap E_1\hookrightarrow E_1\twoheadrightarrow E_1/E_2$$ is non-trivial. This induces a contradiction since $\deg(E')>\deg(E_1/E_2)$. Hence $E'\subset F^*_X(\E)\cap E_2$. In fact $E'=F^*_X(\E)\cap E_2$. Thus $\deg(F^*_X(\E)\cap E_2)=2$.

II. If $\deg(F^*_X(\E)\cap E_2)=2$, then the injection $$F^*_X(\E)\cap E_2\hookrightarrow((F^*_X(\E)\cap E_1)/(F^*_X(\E)\cap E_2))\otimes_{\Ox_X}\Omg^1_X$$ implies that $\deg((F^*_X(\E)\cap E_1)/(F^*_X(\E)\cap E_2))\geq 0$. So $\deg(F^*_X(\E)/(F^*_X(\E)\cap E_1))\leq -2$. Hence the adjoint homomorphism $F^*_X(\E)\rightarrow\Ls$ is not surjective.

III. If $F^*_X(\E)\rightarrow\Ls$ is not surjective, then $\mu_{\mathrm{min}}(F^*_X(\E))\leq -2$. Then by \cite[Theorem 2]{Shatz77} and the classification of Harder-Narasimhan polygons of Frobenius pull backs of Frobenius destabilized semistable vector bundles in the case $(p,g,r,d)=(3,2,3,0)$, we have $\HNP(F^*_X(\E))=\Ps_4$.

In this case, it is easy to see that the Harder-Narasimhan filtration $F^*_X(\E)$ is $$0\subset F^*_X(\E)\cap E_2\subset F^*_X(\E)\cap E_1\subset F^*_X(\E).$$

(2). I. If $\deg(F^*_X(\E)\cap E_2)=1$, then the adjoint homomorphism $F^*_X(\E)\rightarrow\Ls$ is surjective by (1) and $$\mu(F^*_X(\E)\cap E_2)>\mu((F^*_X(\E)\cap E_1)/(F^*_X(\E)\cap E_2))>\mu(F^*_X(\E)/(F^*_X(\E)\cap E_1)).$$ Hence,
$\HNP(F^*_X(\E))=\Ps_3$ and the Harder-Narasimhan filtration of $F^*_X(\E)$ is $$0\subset F^*_X(\E)\cap E_2\subset F^*_X(\E)\cap E_1\subset F^*_X(\E).$$

II. If $\HNP(F^*_X(\E))=\Ps_3$, there exists a unique maximal destabilizing sub-line bundle $E'\subset F^*_X(\E)\cap E_1$ with $\deg(E')=1$. Suppose that $E'\nsubseteq F^*_X(\E)\cap E_2$, then $$E'\hookrightarrow F^*_X(\E)\cap E_1\hookrightarrow E_1\twoheadrightarrow E_1/E_2$$ is non-trivial. This implies $E'\cong E_1/E_2$ since $E'$ and $E_1/E_2$ are line bundles with same degree. Then $E_1=E'\oplus E''$ for some line bundle $E''$ of degree $3$. This induces a contradiction by Corollary \ref{DimofQuot}. Hence $E'\subseteq F^*_X(\E)\cap E_2$. In fact $E'=F^*_X(\E)\cap E_2$, so $\deg(F^*_X(\E)\cap E_2)=1$.

(3). By the proof of (1) and (2), we can conclude that $\HNP(F^*_X(\E))=\Ps_2$ if and only if $\deg(F^*_X(\E)\cap E_2)=0$. In this case, $[\E\hookrightarrow{F_X}_*(\Ls)]\in\Quot^\sharp_X(3,0,\Ls)(k)$ and the Harder-Narasimhan filtration $F^*_X(\E)$ is $$0\subset F^*_X(\E)\cap E_1\subset F^*_X(\E).$$
\end{proof}

\begin{Lemma}[A. Grothendieck, M. Raynaud]\label{Lem:Grothendieck}
Let $k$ be an algebraically closed field of characteristic $p>2$, $X$ a smooth projective curve of genus $g\geq 2$ over $k$, $F:X\rightarrow X_1:=X\times_kk$ the relative Frobenius morphism over $k$. Let $B^1_X$ be the locally free sheaf of locally exact differential forms on $X_1$ defined by the exact sequence of locally free sheaves
$$0\rightarrow\Ox_{X_1}\rightarrow{F}_*(\Ox_X)\rightarrow B^1_X\rightarrow 0.$$
Then the Harder-Narasimhan filtration of $F^*(B^1_X)$ is $$0=V_p\subset V_{p-1}\subset\cdots\subset V_{l+1}\subset V_l\subset\cdots\subset V_1=F^*(B^1_X)$$
such that $V_i/V_{i+1}\cong\mathrm{\Omega}^{\otimes i}_X$ for any $1\leq i\leq p-1$, and $p|(g-1)$ if and only if $$F^*(B^1_X)\cong\mathrm{\Omega}^{\otimes p-1}_X\oplus\mathrm{\Omega}^{\otimes p-2}_X\oplus\cdots\oplus\mathrm{\Omega}^1_X.$$
In particular, in the case $p=3$ and $g=2$, we have $$F^*(B^1_X)\ncong\mathrm{\Omega}^{\otimes 2}_X\oplus\mathrm{\Omega}^1_X.$$
\end{Lemma}

\begin{proof}
Consider the fibred product of $X$ and $X$ over $X_1$
$$\xymatrix{
X\times_{X_1}X \ar@/_/[ddr]_{p_1} \ar@/^/[drr]^{p_2} \ar@{.>}[dr]\\
& X\times_kX \ar[d]^{\pi} \ar[r] & X \ar[d]\\
& X \ar[r]                & \Spec(k).}$$
Let $\triangle$ be the diagonal which is defined by an invertible ideal sheaf $\I$ on $X\times_kX$. Then $X\times_{X_1}X$ is the $(p-2)$-th infinitesimal neighborhood of $\triangle$ in $X\times_kX$, which is defined by ideal sheaf $\I^{p-1}$, and we have a filtration of ideal sheaves
$$0\subset\I^{p-1}\subset\I^{p-2}\subset\cdots\subset\I^2\subset\I\subset\Ox_{X\times_kX}.$$
Taking direct images of above filtration under the first projection $\pi:X\times_kX\rightarrow X$, we can get the filtration
$$0\subset\pi_*(\I^{p-1})\subset\pi_*(\I^{p-2})\subset\cdots\subset\pi_*(\I^2)\subset\pi_*(\I)\subset\pi_*(\Ox_{X\times_kX})$$
such that $\pi_*(\I^i)/\pi_*(\I^{i+1})\cong\mathrm{\Omega}^{\otimes i}_X$ for $1\leq i\leq p-1$ and $\pi_*(\I)\cong F^*(B^1_X)$.
The result of Grothendieck says: the extension of $\pi_*(\I^1)/\pi_*(\I^2)$ by $\pi_*(\I^2)/\pi_*(\I^3)$ corresponds to the element $(g-1)c$ of $H^1(X,\mathrm{\Omega}^1_X)$, where $c$ is the canonical base of $H^1(X,\mathrm{\Omega}^1_X)$. Hence this extension is trivial iff $p$ divides g-1. Then using $H^1(X,\mathrm{\Omega}^i_X)=0$ for $i\geq 2$, we see that, $F^*(B^1_X)$ is a direct sum of $\mathrm{\Omega}^{\otimes i}_X$ for $1\leq i\leq p-1$ iff $p|(g-1)$.
\end{proof}

\begin{Corollary}\label{Cor:NonSplitting}
Let $k$ be an algebraically closed field of characteristic $p>2$, $X$ a smooth projective curve of genus $g\geq 2$ over $k$, $F:X\rightarrow X_1:=X\times_kk$ the relative Frobenius morphism over $k$, and $\Ls$ a line bundle on $X$. Then the Harder-Narasimhan filtration of $F^*_X{F_X}_*(\Ls)$ is $$0=V_p\subset V_{p-1}\subset\cdots\subset V_{l+1}\subset V_l\subset\cdots\subset V_1\subset V_0=F^*F_*(\Ls)$$
such that $V_i/V_{i+1}\cong\mathrm{\Omega}^{\otimes i}_X\otimes\Ls$ for any $0\leq i\leq p-1$, and $p|(g-1)$ if and only if $$F^*_X{F_X}_*(\Ls)\cong(\mathrm{\Omega}^{\otimes p-1}_X\otimes\Ls)\oplus(\mathrm{\Omega}^{\otimes p-2}_X\otimes\Ls)\oplus\cdots\oplus(\mathrm{\Omega}^1_X\otimes\Ls)\oplus\Ls.$$
In particular, in the case $p=3$ and $g=2$, we have $$F^*_X{F_X}_*(\Ls)\ncong(\mathrm{\Omega}^{\otimes 2}_X\otimes\Ls)\oplus(\mathrm{\Omega}^1_X\otimes\Ls)\oplus\Ls.$$
\end{Corollary}

\begin{proof}
Applying Frobenius pull back to the exact sequence of locally free sheaves
$$0\rightarrow\Ox_{X_1}\rightarrow F_*(\Ox_X)\rightarrow B^1_X\rightarrow 0,$$
we have $$0\rightarrow\Ox_X\cong F^*(\Ox_{X_1})\rightarrow F^*F_*(\Ox_X)\rightarrow F^*(B^1_X)\rightarrow 0.$$
On the other hand, it is easy to see that the composition $\Ox_X\cong F^*(\Ox_{X_1})\rightarrow F^*F_*(\Ox_X)\rightarrow\Ox_X$
is an isomorphism. Therefore $F^*F_*(\Ox_X)=V_1\oplus\Ox_X$ and $V_1\cong F^*(B^1_X)$. Thus by Lemma \ref{Lem:Grothendieck} we have $p|(g-1)$ if and only if $$F^*_X{F_X}_*(\Ox_X)\cong\mathrm{\Omega}^{\otimes p-1}_X\oplus\mathrm{\Omega}^{\otimes p-2}_X\oplus\cdots\oplus\mathrm{\Omega}^1_X\oplus\Ox_X.$$
As $F^*F_*(\Ls)\cong F^*F_*(\Ox_X)\otimes\Ls$, so we have $p|(g-1)$ if and only if $$F^*_X{F_X}_*(\Ls)\cong(\mathrm{\Omega}^{\otimes p-1}_X\otimes\Ls)\oplus(\mathrm{\Omega}^{\otimes p-2}_X\otimes\Ls)\oplus\cdots\oplus(\mathrm{\Omega}^1_X\otimes\Ls)\oplus\Ls.$$
\end{proof}

\begin{Proposition}\label{DimofQuot}
Let $k$ be an algebraically closed field of characteristic $3$, $X$ a smooth projective curve of genus $2$ over $k$. Then
$$\Quot_X(3,0,\Pic^{(-1)}(X),\Ps^+_i)=\overline{\Quot_X(3,0,\Pic^{(-1)}(X),\Ps_i)},$$
and $\Quot_X(3,0,\Pic^{(-1)}(X),\Ps^+_i)$ are smooth irreducible projective varieties for $2\leq i\leq 4$,
$$\dim\Quot_X(3,0,\Pic^{(-1)}(X),\Ps^+_i)=\dim\Quot_X(3,0,\Pic^{(-1)}(X),\Ps_i)=
\begin{cases}
5, \mathrm{when}~i=2\\
4, \mathrm{when}~i=3\\
3, \mathrm{when}~i=4\\
\end{cases}$$
\end{Proposition}

\begin{proof}
By \cite{Grothendieck95}, there is a morphism
\begin{eqnarray*}
\Pi:\Quot_X(3,0,\Pic^{(-1)}(X))&\rightarrow&X\times\Pic^{(-1)}(X)\\
{[\E\hookrightarrow{F_X}_*(\Ls)]}&\mapsto&(\Supp({F_X}_*(\Ls)/\E),\Ls).
\end{eqnarray*}
For any point $x\in X$ and any $\Ls\in\Pic^{(-1)}(X)$, we denote the fiber of $\Pi$ over $(x,[\Ls])$ by $\Quot_X(3,0,\Ls,x)$. Then there is a one to one correspondence between the set of closed points $[\E\hookrightarrow{F_X}_*(\Ls)]$ of $\Quot_X(3,0,\Ls,x)$ and the set of $\Ox_x$-submodules $V\subset{F_X}_*(\Ls)_x$ such that ${F_X}_*(\Ls)_x/V\cong k$. The latter has a natural structure of algebraic variety which is isomorphic to projective space $\PP^2_k$. Hence $\Quot_X(3,0,\Pic^{(-1)}(X))$ is a smooth irreducible projective variety of dimension $5$. Without loss of generality, we can assume that $\Ox_x\cong k[[t^3]]$, then ${F_X}_*(\Ls)_x\cong k[[t]]$ endows with $k[[t^3]]$-module structure induced by injection $k[[t^3]]\hookrightarrow k[[t]]$ and $$F^*_X({F_X}_*(\Ls))_x\cong k[[t]]\otimes_{k[[t^3]]}k[[t]].$$

Suppose that the $\Ox_x$-submodule $\E_x$ of ${F_X}_*(\Ls)_x$ corresponds to the $k[[t^3]]$-submodule $V_{\E}$ of $k[[t]]$, then the $\Ox_x$-submodule $F^*_X(\E)_x$ of $F^*_X({F_X}_*(\Ls))_x$ corresponds to the $k[[t]]$-submodule $V_{\E}\otimes_{k[[t^3]]}k[[t]]$ of $k[[t]]\otimes_{k[[t^3]]}k[[t]]$.

Consider the decomposition of $k[[t]]=k[[t^3]]\oplus k[[t^3]]\cdot t\oplus k[[t^3]]\cdot t^2$ as $k[[t^3]]$-module. Then the $k[[t^3]]$-submodule $V_{\E}\subset k[[t]]$ with $k[[t]]/V_{\E}\cong k$ implies that $$k[[t^3]]\cdot t^3\oplus k[[t^3]]\cdot t^4\oplus k[[t^3]]\cdot t^5 \subset V_{\E}.$$

Now, we investigate the intersection of $F^*_X(\E)$ with the canonical filtration $$0\subset E_2\subset E_1\subset F^*_X({F_X}_*(\Ls)).$$
Locally, the stalk ${E_1}_x$ has a basis $\{t\otimes 1-1\otimes t, (t\otimes 1-1\otimes t)^2\}$ and ${E_2}_x$ has a basis $\{(t\otimes 1-1\otimes t)^2\}$ as $k[[t]]$-submodules of $F^*_X({F_X}_*(\Ls))_x\cong k[[t]]\otimes_{k[[t^3]]}k[[t]]$ by \cite[Lemma 3.2]{Sun08i}. Let $[\E\hookrightarrow{F_X}_*(\Ls)]\in\Quot_X(3,0,\Ls,x)(k)$, we claim that
\begin{itemize}
\item[$(a)$] $(t\otimes 1-1\otimes t)^2\notin V_{\E}\otimes_{k[[t^3]]}k[[t]]$
\item[$(b)$] $(t\otimes 1-1\otimes t)^2t\in V_{\E}\otimes_{k[[t^3]]}k[[t]]$ if and only if $\{t,t^2\}\subset V_{\E}$.
\item[$(c)$] $(t\otimes 1-1\otimes t)^2t^2\in V_{\E}\otimes_{k[[t^3]]}k[[t]]$ if and only if $t^2\in V_{\E}$.
\item[$(d)$] $(t\otimes 1-1\otimes t)^2t^3\in V_{\E}\otimes_{k[[t^3]]}k[[t]]$.
\end{itemize}

Suppose that $(t\otimes 1-1\otimes t)^2\in V_{\E}\otimes_{k[[t^3]]}k[[t]]$, then we have $$V_{\E}\otimes_{k[[t^3]]}k[[t]]=(t\otimes 1-1\otimes t)^2k[[t]].$$ It follows that $F^*_X(\E)\cap E_2=E_2$ , so $\deg(F^*_X(\E)\cap E_2)=3$. This contradicts to Proposition \ref{Prop:Intersection}.

Since
\begin{eqnarray*}
(t\otimes 1-1\otimes t)^2t&=&t^2\otimes t-2t\otimes t^2+1\otimes t^3\\
&=&t^2\otimes t-2t\otimes t^2+t^3\otimes 1
\end{eqnarray*}
and $\{t^3\}\subset V_{\E}$ by $k[[t^3]]\cdot t^3\oplus k[[t^3]]\cdot t^4\oplus k[[t^3]]\cdot t^5 \subset V_{\E}$, then we have $$(t\otimes 1-1\otimes t)^2t\in V_{\E}\otimes_{k[[t^3]]}k[[t]]~\mathrm{if~and~only~if}~t^2\otimes t-2t\otimes t^2\in V_{\E}\otimes_{k[[t^3]]}k[[t]],$$ which is equivalent to $\{t,t^2\}\subset V_{\E}$.

Since
\begin{eqnarray*}
(t\otimes 1-1\otimes t)^2t^2&=&t^2\otimes t^2-2t\otimes t^3+1\otimes t^4\\
&=&t^2\otimes t^2-2t^4\otimes 1+t^3\otimes t
\end{eqnarray*}
and $\{t^3,t^4\}\subset V_{\E}$ by $k[[t^3]]\cdot t^3\oplus k[[t^3]]\cdot t^4\oplus k[[t^3]]\cdot t^5 \subset V_{\E}$, then we have $$(t\otimes 1-1\otimes t)^2t^2\in V_{\E}\otimes_{k[[t^3]]}k[[t]]~\mathrm{if~and~only~if}~t^2\otimes t^2\in V_{\E}\otimes_{k[[t^3]]}k[[t]],$$ which is equivalent to $x^2\in V_{\E}$.

Since
\begin{eqnarray*}
(t\otimes 1-1\otimes t)^2t^3&=&t^2\otimes t^3-2t\otimes t^4+1\otimes t^5\\
&=&t^5\otimes 1-2t^4\otimes t+t^3\otimes t^2
\end{eqnarray*}
and $\{t^3,t^4,t^5\}\subset V_{\E}$ by $k[[t^3]]\cdot t^3\oplus k[[t^3]]\cdot t^4\oplus k[[t^3]]\cdot t^5 \subset V_{\E}$. It follows that $(t\otimes 1-1\otimes t)^2t^3\in V_{\E}\otimes_{k[[t^3]]}k[[t]]$.

In summary, by above claim, we have $$1\leq\dim {E_2}_x/((V_{\E}\otimes_{k[[t^3]]}k[[t]])\cap {E_2}_x)\leq 3.$$
Precisely,
$$\dim {E_2}_x/((V_{\E}\otimes_{k[[t^3]]}k[[t]])\cap {E_2}_x)=
\begin{cases}
1& \text{if ond only if}~\{t,t^2\}\subset V_{\E}\\
2& \text{if ond only if}~t\notin V_{\E}~\text{and}~t^2\in V_{\E}\\
3& \text{if ond only if}~t^2\notin V_{\E}
\end{cases}$$

Consider the exact sequence of $\Ox_X$-modules
$$0\rightarrow F^*_X(\E)\cap E_2\rightarrow E_2\rightarrow E_2/(F^*_X(\E)\cap E_2)\rightarrow0.$$
Notice that $E_2/(F^*_X(\E)\cap E_2)={E_2}_x/((V_{\E}\otimes_{k[[t^3]]}k[[t]])\cap{E_2}_x)$. Therefore, by Proposition \ref{Prop:Intersection}, we have
\begin{eqnarray*}
\HNP(F^*_X(\E))=\Ps_2 &\Leftrightarrow& \deg(F^*_X(\E)\cap E_2)=0\\
&\Leftrightarrow& \deg({E_2}_x/((V_{\E}\otimes_{k[[t^3]]}k[[t]])\cap{E_2}_x))=3\\
&\Leftrightarrow& t^2\notin V_{\E}.
\end{eqnarray*}

\begin{eqnarray*}
\HNP(F^*_X(\E))=\Ps_3 &\Leftrightarrow& \deg(F^*_X(\E)\cap E_2)=1\\
&\Leftrightarrow& \deg({E_2}_x/((V_{\E}\otimes_{k[[t^3]]}k[[t]])\cap{E_2}_x))=2\\
&\Leftrightarrow& t\notin V_{\E}~\text{and}~t^2\in V_{\E}.
\end{eqnarray*}

\begin{eqnarray*}
\HNP(F^*_X(\E))=\Ps_4 &\Leftrightarrow& \deg(F^*_X(\E)\cap E_2)=2\\
&\Leftrightarrow& \deg({E_2}_x/((V_{\E}\otimes_{k[[t^3]]}k[[t]])\cap{E_2}_x))=1\\
&\Leftrightarrow& \{t,t^2\}\subset V_{\E}.
\end{eqnarray*}

For $i=2,3,4$, let $\Quot_X(3,d,\Ls_x,\Ps^+_i(d))$ be the closed subschemes of $\Quot_X(3,d,\Ls_x)$ consisting of closed points $$\Quot_X(3,0,\Ls,x,\Ps^+_i)(k)=\{~[\E\hookrightarrow{F_X}_*(\Ls)]\in\Quot_X(3,0,\Ls,x)(k)~|~\HNP(F^*_X(\E))\succcurlyeq\Ps_i~\}$$
Then
\begin{eqnarray*}
\Quot_X(3,0,\Ls,x,\Ps^+_2)&\cong&\{V|~k[[t^3]]\text{-submodule}~V\subset k[[t]],k[[t]]/V\cong k\}\cong\PP^2_k\\
\Quot_X(3,0,\Ls,x,\Ps^+_3)&\cong&\{V|~k[[t^3]]\text{-submodule}~V\subset k[[t]],k[[t]]/V\cong k,t^2\in V_{\E}\}\cong\PP^1_k\\
\Quot_X(3,0,\Ls,x,\Ps^+_4)&\cong&\{V|~k[[t^3]]\text{-submodule}~V\subset k[[t]],k[[t]]/V\cong k,\{t,t^2\}\subset V_{\E}\}\cong\{p\}.
\end{eqnarray*}
So $\Quot_X(3,0,\Pic^{(-1)}(X),\Ps^+_i)$ are smooth irreducible projective varieties for $2\leq i\leq 4$, and
$$\dim\Quot_X(3,0,\Pic^{(-1)}(X),\Ps^+_i)=
\begin{cases}
5, \mathrm{when}~i=2\\
4, \mathrm{when}~i=3\\
3, \mathrm{when}~i=4\\
\end{cases}$$
Since $\Quot_X(3,0,\Pic^{(-1)}(X),\Ps_i)$ is an open subvariety of $\Quot_X(3,0,\Pic^{(-1)}(X),\Ps^+_i)$ for any $2\leq i\leq 4$, we have $$\Quot_X(3,0,\Pic^{(-1)}(X),\Ps^+_i)=\overline{\Quot_X(3,0,\Pic^{(-1)}(X),\Ps_i)},$$ $$\dim\Quot_X(3,0,\Pic^{(-1)}(X),\Ps^+_i)=\dim\Quot_X(3,0,\Pic^{(-1)}(X),\Ps_i).$$
\end{proof}

\section{Frobenius Stratification of Moduli Spaces}\

We first introduce some notations which will be used in this section. Let $k$ be an algebraically closed field of characteristic $p>0$, $X$ a smooth projective curve of genus $g$ over $k$. Let $r$, $d$ and $t$ be integers with $r>0$, and $\Ps\in\ConPgn(r,pd)$. Then we define the subschemes of $\Quot_X(3,0,\Pic^{(-1)}(X))$ and $\Quot^\sharp_X(r,d,\Pic^{(t)}(X),\Ps)$ in the sense of closed points as the following:
\begin{eqnarray*}
\Quot^s_X(r,d,\Pic^{(t)}(X))(k)&:=&\{~[\E\hookrightarrow{F_X}_*(\Ls)]\in\Quot_X(r,d,\Pic^{(t)}(X))(k)~|~\E~\text{is stable}~\}\\
\Quot^s_X(r,d,\Pic^{(t)}(X),\Ps)(k)&:=&\{~[\E\hookrightarrow{F_X}_*(\Ls)]\in\Quot_X(r,d,\Pic^{(t)}(X),\Ps)(k)~|~\E~\text{is stable}~\}.\\
\Quot^s_X(r,d,\Pic^{(t)}(X),\Ps^+)(k)&:=&\{~[\E\hookrightarrow{F_X}_*(\Ls)]\in\Quot_X(r,d,\Pic^{(t)}(X),\Ps^+)(k)~|~\E~\text{is stable}~\}\\
\Quot^{s,\sharp}_X(r,d,\Pic^{(t)}(X))(k)&:=&\{~[\E\hookrightarrow{F_X}_*(\Ls)]\in\Quot^\sharp_X(r,d,\Pic^{(t)}(X))(k)~|~\E~\text{is stable}~\}\\
\Quot^{s,\sharp}_X(r,d,\Pic^{(t)}(X),\Ps)(k)&:=&\{~[\E\hookrightarrow{F_X}_*(\Ls)]\in\Quot^\sharp_X(r,d,\Pic^{(t)}(X),\Ps)(k)~|~\E~\text{is stable}~\}.\\
\Quot^{s,\sharp}_X(r,d,\Pic^{(t)}(X),\Ps^+)(k)&:=&\{~[\E\hookrightarrow{F_X}_*(\Ls)]\in\Quot^\sharp_X(r,d,\Pic^{(t)}(X),\Ps^+)(k)~|~\E~\text{is stable}~\}
\end{eqnarray*}

We now study the geometric properties of Frobenius strata in the Frobenius stratification of moduli space $\M^s_X(3,0)$, where $X$ is a smooth projective curve of genus $2$ over an algebraically closed field $k$ of characteristic $3$.

In the case $(p,g,r,d)=(3,2,3,0)$, by Proposition \ref{Prop:Subsheaf}, there is a proper morphism
\begin{eqnarray*}
\theta:\Quot_X(3,0,\Pic^{(-1)}(X))&\rightarrow&\M^{ss}_X(3,0)\\
{[\E\hookrightarrow{F_X}_*(\Ls)]}&\mapsto&[\E].
\end{eqnarray*}

By restricting the morphism $\theta$ to the stable locus $\Quot^s_X(3,0,\Pic^{(-1)}(X))$ of $\Quot_X(3,0,\Pic^{(-1)}(X))$, then we have a morphism from $\Quot^s_X(3,0,\Pic^{(-1)}(X))$ to $\M^s_X(3,0)$, denoted by $\theta^s$. Hence, $\theta^s$ is a proper morphism.

\begin{Proposition}\label{Prop:InjMorphism}
Let $k$ be an algebraically closed field of characteristic $3$, $X$ a smooth projective curve of genus $2$ over $k$. Then the image of the morphism
\begin{eqnarray*}
\theta^s:\Quot^s_X(3,0,\Pic^{(-1)}(X))&\rightarrow&\M^s_X(3,0)\\
{[\E\hookrightarrow{F_X}_*(\Ls)]}&\mapsto&[\E]
\end{eqnarray*}
is the subset
$$\{~[\E]\in\M^s_X(3,0)(k)~|~\HNP(F^*_X(\E))\in\{\Ps_2,\Ps_3,\Ps_4\}~\}.$$
Moreover, the restriction $\theta^s|_{\Quot^{s,\sharp}_X(3,0,\Pic^{(-1)}(X))}$ is an injective morphism and the image of $\theta^s|_{\Quot^{s,\sharp}_X(3,0,\Pic^{(-1)}(X))}$ is the subset
$$\{~[\E]\in\M^s_X(3,0)(k)~|~\HNP(F^*_X(\E))\in\{\Ps_2,\Ps_3\}~\}.$$
\end{Proposition}

\begin{proof}
Let $[\E\hookrightarrow{F_X}_*(\Ls)]\in\Quot_X(3,0,\Ls)(k)$, then $\HNP(F^*_X(\E))\in\{\Ps_2,\Ps_3,\Ps_4\}$ by Proposition \ref{Prop:Intersection}. It follows that the image of $\theta$ lies in the following subset $$\{~[\E]\in\M^s_X(3,0)(k)~|~\HNP(F^*_X(\E))\in\{\Ps_2,\Ps_3,\Ps_4\}~\}.$$

On the other hand, let $[\E]\in\M^s_X(3,0)(k)$ such that $\HNP(F^*_X(\E))\in\{\Ps_2,\Ps_3,\Ps_4\}$. Then $F^*_X(\E)$ has a quotient line bundle $\Ls'$ of $\deg(\Ls')\leq-1$. Embedding $\Ls'$ into some line bundle $\Ls$ of $\deg(\Ls)=-1$, then we have the non-trivial homomorphism $$F^*_X(\E)\twoheadrightarrow\Ls'\hookrightarrow\Ls.$$ Then the adjoint homomorphism $\E\hookrightarrow{F_X}_*(\Ls)$ is an injection by Proposition \ref{Prop:Injection}. Hence, the image of $\theta$ is $$\{~[\E]\in\M^s_X(3,0)(k)~|~\HNP(F^*_X(\E))\in\{\Ps_2,\Ps_3,\Ps_4\}~\}.$$

Now, we will prove $\theta|_{\Quot^{s,\sharp}_X(3,0,\Pic^{(-1)}(X))}$ is an injective morphism. Let $$e_i:=[\E_i\hookrightarrow{F_X}_*(\Ls_i)]\in\Quot^{s,\sharp}_X(3,0,\Pic^{(-1)}(X))(k),$$ where $\Ls_i\in\Pic^{(-1)}(X)$, $i=1,2$. Suppose that $\theta(e_1)=\theta(e_2)\in\M^s_X(3,0)(k)$, i.e. $\E_1\cong\E_2$. Since $\HNP(F^*_X(\E))\in\{\Ps_2,\Ps_3\}$, we have $\mu_{\mathrm{min}}(F^*_X(\E_i))=-1$. So the surjection $F^*_X(\E_i)\rightarrow\Ls_i$ implies that $\Ls_i$ is the quotient line bundle of $F^*_X(\E_i)$ with minimal slope in the Harder-Narasimhan filtration of $F^*_X(\E_i)$ for $i=1,2$. By the uniqueness of Harder-Narasimhan filtration, there exists an isomorphism $\psi:\Ls_1\rightarrow\Ls_2$ making the following diagram
$$\xymatrix{
  F^*_X(\E_1) \ar[r]\ar[d]_{\phi}^{\cong} & \Ls_1 \ar[r]\ar[d]_{\psi}^{\cong} & 0\\
  F^*_X(\E_2) \ar[r] & \Ls_2 \ar[r] & 0}$$
commutative, where the isomorphism $\phi$ is induced from an isomorphism $\E_1\stackrel{\cong}{\rightarrow}\E_2$. By adjunction, we have the commutative diagram
$$\xymatrix{
  0 \ar[r] & \E_1 \ar[r]\ar[d]^{\cong} & {F_X}_*(\Ls_1) \ar[d]_{{F_X}_*(\psi)}^{\cong}\\
  0 \ar[r] & \E_2 \ar[r] & {F_X}_*(\Ls_2)}$$
where the vertical homomorphism is the isomorphism $${F_X}_*(\psi):{F_X}_*(\Ls_1)\stackrel{\cong}{\rightarrow}{F_X}_*(\Ls_2).$$ This implies $\E_1=\E_2$ as subsheaves of ${F_X}_*(\Ls)$, where $$[\Ls]=[\Ls_1]=[\Ls_2]\in\Pic^{(-1)}(X)(k).$$ Thus, $e_1$ and $e_2$ are the some point in the $\Quot^{s,\sharp}_X(3,0,\Pic^{(-1)}(X))$. Hence the morphism $\theta|_{\Quot^{s,\sharp}_X(3,0,\Pic^{(-1)}(X))}$ is injective.

Let $[\E]\in\M^s_X(3,0)(k)$ and $\HNP(F^*_X(\E))\in\{\Ps_2,\Ps_3\}$. Then $F^*_X(\E)$ has a quotient line bundle $\Ls$ of $\deg(\Ls)=-1$. Then by Proposition \ref{Prop:Injection}, the adjoint homomorphism $\E\hookrightarrow{F_X}_*(\Ls)$ is an injective homomorphism. Therefore $$e:=[\E\hookrightarrow{F_X}_*(\Ls)]\in\Quot^{s,\sharp}_X(3,0,\Pic^{(-1)}(X))(k)$$ and $\theta(e)=[\E]$. Hence the image of $\theta|_{\Quot^{s,\sharp}_X(3,0,\Pic^{(-1)}(X))}$ is the subset
$$\{~[\E]\in\M^s_X(3,0)(k)~|~\HNP(F^*_X(\E))\in\{\Ps_2,\Ps_3\}~\}.$$
\end{proof}

We can obtain the geometric properties of Frobenius strata in the Frobenius stratification of moduli space $\M^s_X(3,0)$ from the geometric properties of Frobenius strata in $\Quot^s_X(3,0,\Pic^{(-1)}(X))$, where $X$ is a smooth projective curve of genus $2$ over an algebraically closed field $k$ of characteristic $3$.

\begin{Theorem}\label{Thm:FrobStra}
Let $k$ be an algebraically closed field of characteristic $3$, $X$ a smooth projective curve of genus $2$ over $k$. Then
\begin{itemize}
    \item[$(1)$] $S_X(3,0,\Ps^+_1)\cong S_X(3,0,\Ps^+_2)$, $S_X(3,0,\Ps_1)\cong S_X(3,0,\Ps_2)$, and $$S_X(3,0,\Ps^+_1)\cap S_X(3,0,\Ps^+_2)=S_X(3,0,\Ps^+_3).$$
    \item[$(2)$] $S_X(3,0,\Ps^+_i)=\overline{S_X(3,0,\Ps_i)}$, $S_X(3,0,\Ps_i)$ and $S_X(3,0,\Ps^+_i)$ are irreducible quasi-projective varieties for $1\leq i\leq 4$, and $$\dim S_X(3,0,\Ps^+_i)=\dim S_X(3,0,\Ps_i)=
\begin{cases}
5, \mathrm{when}~i=1\\
5, \mathrm{when}~i=2\\
4, \mathrm{when}~i=3\\
2, \mathrm{when}~i=4\\
\end{cases}$$
\end{itemize}
\end{Theorem}

\begin{proof}
(1). The morphism $$\iota:\M^s_X(3,0)\rightarrow\M^s_X(3,0)$$ $$[\E]\mapsto[\E^\vee]$$ induces an involution of $\M^s_X(3,0)$, which maps $S_X(3,0,\Ps^+_1)$ (resp. $S_X(3,0,\Ps_1)$) onto $S_X(3,0,\Ps^+_2)$ (resp. $S_X(3,0,\Ps_2)$). Hence we have $$S_X(3,0,\Ps^+_1)\cong S_X(3,0,\Ps^+_2),$$ $$S_X(3,0,\Ps_1)\cong S_X(3,0,\Ps_2).$$ The fact $S_X(3,0,\Ps^+_1)\cap S_X(3,0,\Ps^+_2)=S_X(3,0,\Ps^+_3)$ is followed by the classification of Harder-Narasimhan polygons of Frobenius pull backs of Frobenius destabilized stable vector bundles in the case $(p,g,r,d)=(3,2,3,0)$.

(2). By Proposition \ref{Prop:Subsheaf} and the openness of stability of vector bundles, we have $\Quot^s_X(3,0,\Pic^{(-1)}(X),\Ps^+_i)$ are non-empty open subvarieties of $\Quot_X(3,0,\Pic^{(-1)}(X),\Ps^+_i)$ for $2\leq i\leq 4$. Hence, $\Quot^s_X(3,0,\Pic^{(-1)}(X),\Ps^+_i)$ are smooth quasi-projective varieties for $2\leq i\leq 4$.

The morphism
\begin{eqnarray*}
\theta^s:\Quot^s_X(3,0,\Pic^{(-1)}(X))&\rightarrow&\M^s_X(3,0)\\
{[\E\hookrightarrow{F_X}_*(\Ls)]}&\mapsto&[\E]
\end{eqnarray*}
maps $\Quot^s_X(3,0,\Pic^{(-1)}(X),\Ps^+_i)$ onto $S_X(3,0,\Ps^+_i)$ for $2\leq i\leq 4$. Then by Proposition \ref{DimofQuot} and properness of $\theta^s$, we have $S_X(3,0,\Ps^+_i)$ are irreducible quasi-projective varieties for $2\leq i\leq 4$. Since $S_X(3,0,\Ps_i)$ is an open subvariety of $S_X(3,0,\Ps^+_i)$, we have $$S_X(3,0,\Ps^+_i)=\overline{S_X(3,0,\Ps_i)}$$ for any $2\leq i\leq 4$. Moreover, by Proposition \ref{Prop:InjMorphism}, the injection $\theta^s|_{\Quot^{s,\sharp}_X(3,0,\Pic^{(-1)}(X))}$ maps $\Quot^s_X(3,0,\Pic^{(-1)}(X),\Ps_i)$ onto $S_X(3,0,\Ps_i)$ for $i=2,3$. Then by Proposition \ref{DimofQuot}, we have
$$\dim S_X(3,0,\Ps^+_i)=\dim S_X(3,0,\Ps_i)=\dim\Quot^s_X(3,0,\Pic^{(-1)}(X),\Ps_i)=
\begin{cases}
5& i=2\\
4& i=3
\end{cases}$$
As $S_X(3,0,\Ps^+_1)\cong S_X(3,0,\Ps^+_2)$, so $S_X(3,d,\Ps^+_1)=\overline{S_X(3,0,\Ps^+_1)}$ is an irreducible quasi-projective variety and $\dim S_X(3,d,\Ps^+_1)=\dim S_X(3,d,\Ps_1)=5$.

Now we study the properties of stratum $S_X(3,0,\Ps_4)$. Any $[\E]\in S_X(3,0,\Ps_4)$ has the form ${F_X}_*(\Ls')$ for some line bundle $\Ls'$ of $\deg(\Ls')=-2$ by \cite[Lemma 3.1]{Li14} (or Lemma \ref{Lemma:CanPolygon}). Moreover, by Lemma \ref{FrobImage}, the morphism
\begin{eqnarray*}
P^s_{\Frob}:\M^s_X(1,-2)&\rightarrow&\M^s_X(3,0)\\
{[\Ls']}&\mapsto&[{F_X}_*(\Ls')]
\end{eqnarray*}
is a closed immersion and the image of $P^s_{\Frob}$ is just the $S_X(3,0,\Ps_4)$. Thus $S_X(3,0,\Ps_4)=S_X(3,0,\Ps^+_4)$ is isomorphic to Jacobian variety $\mathrm{Jac}_X$ of $X$ which is a smooth irreducible projective variety of dimension $2$.
\end{proof}

\begin{Lemma}[L. Li \cite{Li14}]\label{FrobImage}
Let $k$ be an algebraically closed field of characteristic $p>0$, $X$ a smooth projective curve of genus $g\geq 2$ over $k$. Then the well-defined set-theoretic map
\begin{eqnarray*}
P^s_{\Frob}:\M^s_X(r,d)&\rightarrow&\M^s_X(rp,d+r(p-1)(g-1))\\
{[\E]}&\mapsto&[{F_X}_*(\E)]
\end{eqnarray*}
is a closed immersion.
\end{Lemma}

For any integer $d$ with $3|d$, we can construct an isomorphism of $$\M^s_X(3,d)\cong\M^s_X(3,0)$$ by tensoring any given line bundle of degree $\frac{d}{3}$. Then the Frobenius stratification of $\M^s_X(3,d)$ is easily deduced from $\M^s_X(3,0)$, when $X$ is a smooth projective curve of genus $2$ over an algebraically closed field $k$ of characteristic $3$.

Now, we will study the geometric properties of a specific Frobenius stratum in the moduli spaces of stable vector bundles of higher rank. We first generalize \cite[Lemma 3.1]{Li14} to the higher rank case.

Fix a quadruple $(p,g,r,d)$, we construct a convex polygon $\Ps^{\mathrm{can}}_{rp,d}\in\ConPgn(r,pd)$ as follows:
$$\Ps^{\mathrm{can}}_{rp,d}:~\text{with vertexes}~\Big(i\cdot r,d\cdot i+r\cdot i(p-i)(g-1)\Big)~\text{for}~0\leq i\leq p.$$
It is easy to check that the difference between the slopes of two successive line segments is $2g-2$.

\begin{Lemma}[C. Liu, M. Zhou \cite{LiuZhou13}]\label{Lemma:CanPolygon}
Let $k$ be an algebraically closed field of characteristic $p>0$, $X$ a smooth projective curve of genus $g\geq 2$ over $k$. Let $\E$ be a stable vector bundle of rank $rp$ on $X$. Then the following statements are equivalent:
\begin{itemize}
    \item[$(i)$] $\HNP(F^*_X(\E))=\Ps^{\mathrm{can}}_{rp,d}$.
    \item[$(ii)$] $\mu_{\mathrm{max}}(F^*_X(\E))-\mu_{\mathrm{min}}(F^*_X(\E))=(p-1)(2g-2)$.
    \item[$(iii)$] There exists a stable vector bundle $\F$ such that $\E={F_X}_*(\F)$.
\end{itemize}
\end{Lemma}

\begin{Theorem}\label{GeneralStratum}
Let $k$ be an algebraically closed field of characteristic $p>0$, $X$ a smooth projective curve of genus $g\geq 2$ over $k$. Then the subset $$V_{rp,d}=\{[\E]\in\M^s_X(rp,d)(k)~|~\HNP(F^*_X(\E))=\Ps^{\mathrm{can}}_{rp,d}\}$$
is a smooth irreducible closed subvariety of dimension $r^2(g-1)+1$ in $\M^s_X(rp,d)$.
\end{Theorem}
\begin{proof}
By Lemma \ref{Lemma:CanPolygon}, we know that $V_{rp,d}$ is precisely the image of the morphism $P^s_{\Frob}:\M^s_X(r,d-r(p-1)(g-1))\rightarrow\M^s_X(rp,d)$.
Then this Theorem is followed by Lemma \ref{FrobImage}.
\end{proof}

\section*{Acknowledgements}
I would like to thank Luc Illusie, Xiaotao Sun, Yuichiro Hoshi, Mingshuo Zhou, Junchao Shentu, Yongming Zhang for their interest and helpful conversations. Yifei Chen gives me a lot of help to improve the presentation of my manuscript. Especially, I express my greatest appreciation to the late Professor Michel Raynaud, who tells me Lemma \ref{Lem:Grothendieck} according to an unpublished note of Alexander Grothendieck.

\bibliographystyle{amsplain}

\end{document}